\documentclass[10pt,twoside]{article}
\usepackage{amsmath,amsbsy,amsfonts}
\usepackage{graphicx,color}
\everymath{\displaystyle}
\usepackage[notref,notcite]{showkeys}
\newtheorem{theorem}{Theorem}[section]
\newtheorem{lemma}[theorem]{Lemma}
\newtheorem{proposition}[theorem]{Proposition}
\newtheorem{corollary}[theorem]{Corollary}

\newtheorem{claim}[theorem]{Claim}

\newenvironment{proof}[1][Proof]{\noindent\textbf{#1.} }{\ \rule{0.5em}{0.5em}}

\newcommand{\ds}{\displaystyle}
\newcommand{\R}{\mathbb{R}}

\begin{document}

\title{\bf Existence of multi-bump solutions to biharmonic operator with critical exponential growth in $\mathbb{R}^4$.}

\author{Al\^{a}nnio B. N\'{o}brega \thanks{alannio@dme.ufcg.edu.br}\, \, \ Denilson S. Pereira\thanks{denilsonsp@mat.ufcg.edu.br}\ \vspace{2mm}
	\and {\small  Universidade Federal de Campina Grande} \\ {\small Unidade Acad\^emica de Matem\'{a}tica} \\ {\small CEP: 58429-900, Campina Grande - Pb, Brazil}\\}

\date{}
\maketitle

\begin{abstract}
	Using variational methods, we establish existence of multi-bump solutions for the
	following class of problems 
	$$
	\left\{
	\begin{array}{l}
	\Delta^2 u  +(\lambda V(x)+1)u  =  f(u), \quad \mbox{in} \quad \mathbb{R}^{4},\\
	u \in H^{2}(\mathbb{R}^{4}),
	\end{array}
	\right.
	$$
	where $\Delta^2$ is the biharmonic operator, $f$ is a continuous
	function with critical exponential growth and $V : \mathbb{R}^4 \rightarrow \mathbb{R}$ is a continuous function
	verifying some conditions.

	\vspace{0.3cm}
	
	
	\vspace{0.3cm}
	
\end{abstract}

\section{Introduction}

In this paper,  we are concerned with the existence and multiplicity of multi-bump type solutions for the following class of problems
$$
\left\{ \begin{array}{cc}
\Delta^2 u  +(\lambda V(x)+1)u  & =  f(u), \quad \mbox{in} \quad \mathbb{R}^{4}, \\[0.3cm]
u \in H^2(\mathbb{R}^4);\,\,\,\,\,\,\,\,\,\,\,\,\,\,\,\,&  
\end{array}
\right.\leqno{(P_{\lambda})}
$$
where $\Delta^{2}$ denotes the biharmonic operator,  $\lambda\in(0,\infty)$, $V:\mathbb{R}^4\to\mathbb{R}$ is a continuous and nonnegative potential such that $\Omega:=int\, V^{-1}(0)$ has $k$ connected components denoted by $\Omega_j$, $j\in\{1,\cdots, k\}$  and $f:\mathbb{R}\to\mathbb{R}$ is a continuous function with critical exponential growth, that is, $f$ behaves like $exp(\alpha_0s^2)$ at infinity, for some $\alpha_0>0$. Namely,  there exists $\alpha_0>0$ such that
\begin{equation} \label{expg}
\lim_{|s|\rightarrow\infty}\dfrac{f(s)}{e^{\alpha|s|^2}}=0\ \
\forall \alpha>\alpha_0,\ \
\lim_{|s|\rightarrow\infty}\dfrac{f(s)}{e^{\alpha|s|^2}}=\infty\ \
\forall\alpha<\alpha_0 \,\,\, ( \mbox{see} \, \cite{DMR} \, ).
\end{equation}

In the last years, problems involving the biharmonic operator  has been studied by many researchers, in part because this operator helps to describe the mechanical vibrations of an elastic plate, which among other things describes the traveling waves in a suspension bridge, see \cite{BDF,FG,G,GK,LM}. On the other hand, the biharmonic draws attention by the difficulties encountered when trying to adapt known results for the Laplacian, for example, we cannot always rely on a maximum principle, and also, if $u$ belongs to $ H ^ 2 (\mathbb{R}^4)$,  we cannot claim that $u ^ \pm$ belong to $ H ^ 2 (\mathbb{R}^4)$. 

Concerned multi-bump type solutions, Ding and Tanaka in \cite{D&T} have considered the problem
$$\left\{ \begin{array}{l}
-\Delta u  +(\lambda V(x)+Z(x))u  = u^p, \quad \mbox{in} \quad \mathbb{R}^{N}, \\
u>0,\ \mbox{in}\ \mathbb{R}^N,
\end{array}
\right.\leqno{(\tilde{P}_{\lambda})}$$
In that paper, they showed that problem $(\tilde{P}_{\lambda})$ has at least $2^k-1$ multi-bump type solutions, for $\lambda$ large enough.The same type of result was
obtained by Alves, de Morais Filho and Souto in \cite{A-M-S} and
Alves and Souto \cite{AS}, by assuming that $f$ has a critical
growth for the case $N\geq 3$ and critical exponential  growth when $N=2$, respectively. More recently, Alves and Pereira, in \cite{AP} showed the existence of nodal multi-bump solutions to a same type of problem, but considering a general nonlinearity with critical exponential growth. Related to the existence
of multi-bump solutions for the biharmonic problem $(P_{\lambda})$, Alves and N\'obrega \cite{AN}, developed a technique inspired in the works by Bartsch and Wang \cite{B&W1, B&W2} to show the existence of multi-bump solution. This required, because  it is not clear that the celebrated method developed by del Pino and Felmer \cite{Del&Felm}  can be used to  problems involving biharmonic operator. 

In this work, we intend to prove the existence and multiplicity of multi-bump positive solutions for problem $(P_\lambda)$.  Our main result completes the studies made in the previous papers, because we will work with the biharmonic operator and nonlinearity with a critical exponential growth. Since we will work with critical exponential growth in whole $\mathbb{R}^{4}$, a key point in our arguments is the following Adams-type inequality due to B. Ruf and F. Sani \cite{RS}:

\begin{equation} \label{X2}
\sup_{u\in H^2(\mathbb{R}^4),~\|u\|_{H^2}\leq 1}\int_{\mathbb{R}^4}(e^{\alpha u^2}-1)dx\left\{
\begin{array}{cc}
<+\infty~\mbox{for}~\alpha\leq 32\pi^2,\\ [0.3cm]
=+\infty~\mbox{for}~\alpha> 32\pi^2,
\end{array}
\right.
\end{equation}
where $\|u\|^2_{H^2}:=\|(-\Delta+I)u\|_2^2=\|\Delta u\|_2^2+2\|\nabla u\|_2^2+\|u\|_2^2$.

Next we shall describe the conditions on the functions $V(x)$ and $f(x,s)$ in a more precise way. The  potential potential $V:\mathbb{R}^4 \rightarrow\mathbb{R}$ satisfies
\begin{description}
	\item[$(V_1)$]  $V(x)\geq 0, \ \forall \ x\in \mathbb{R}^4$;
	\item[$(V_2)$]  $\Omega= int V^{-1}(\{0\})$ is a non-empty bounded open set with smooth boundary, consisting of $k$ connected components, more precisely,
	\begin{itemize}
		\item $\Omega=\bigcup_{j=1}^{k}\Omega_j ;$
		\item $dist(\Omega_i,\Omega_j)>0,\ i\neq j.$
	\end{itemize} 
	\item[$(V_3)$]  There is $M_0>0$ such that $ |\{ x\in \mathbb{R}^N ;\, V(x)\leq
	M_0\}|<+\infty.$
\end{description}
Hereafter, if $A \subset \mathbb{R}^{4}$ is a mensurable set, $|A|$ denotes its Lebesgue's measure.

The natural space to treat the problem $(P_\lambda)$ with variational method is
$$
E_{\lambda}=\left\{ u\in H^2(\mathbb{R}^4);\ \int_{\mathbb{R}^4}V(x)\left|u\right|^2 dx <+\infty\right\}.
$$
The subspace $E_{\lambda}$ endowed with the inner product
$$
(u,v)_{\lambda}= \int_{\mathbb{R}^4} \left[\Delta u\Delta v+(\lambda V(x)+1)uv\right]dx,
$$
is a Hilbert space and the norm generated by this inner product will be denoted by $\|\cdot\|_{\lambda}$. The space $E_{\lambda}$ is continuously embedding in $H^2(\mathbb{R}^4),$ that is, there exists $\overline{C}>0$, independent of $\lambda$, such that
\begin{equation}\label{imb}
\|u\|_{H^2}\le \overline{C}\|u\|_{\lambda}, \ \forall u \in E_{\lambda}.
\end{equation}

Related to nonlinearity, we assume that $f:\mathbb{R} \rightarrow \mathbb{R}$ is a $C^1$-function verifying  the following hypotheses:
\begin{description}
	
\item[$(f_1)$] There is $C>0$ such that
$$
|f(s)|,|f'(s)|\leq Ce^{32\pi^2 |s|^2}\ \ \mbox{for all}\ \ s\in\R;
$$

\item[$(f_2)$] $\ds\lim_{s\rightarrow
0}\dfrac{f(s)}{s}=0$;

\item[$(f_3)$] There is $\theta>2$ such that
$$0<\theta F(s):=\theta\int_0^{s}f(t)dt\leq sf(s),\ \ \mbox{for all}\ \
s\in\R\setminus\{0\}.$$

\item[$(f_4)$] The function $s\mapsto\dfrac{f(s)}{|s|}$ is
strictly increasing for $s\neq 0$.

\item[$(f_5)$] There exist constants $p>2$ and $C_p>0$ such  that
$$
f(s)\geq C_p|s|^{p-1}\ \ \mbox{for all}\ \
s\in\R
$$
with
$$
C_p>\left[\dfrac{2\overline{C}^2 k\theta(p-2)}{p(\theta-2)}\right]^{(p-2)/2}S_p^p,~~ \mbox{where}~~ S_p=\displaystyle\max_{1\leq j\leq k}\gamma_j
$$
and
$$\gamma_j=\displaystyle\inf_{u\in H_0^2(\Omega_j)}\dfrac{\left(\int_{\Omega_j}(|\Delta u|^2+u^2)dx\right)^{1/2}}{\left(\int_{\Omega_j}|u|^pdx\right)^{1/p}}.$$

\end{description}

It is easily seen that $(f_1)-(f_5)$ hold for nonlinearities of the form
$$
f(s)=\left(pC_p|s|^{p-2}s+64\pi^2 C_p|s|^ps\right)e^{\alpha s^2}, \quad \mbox{for} \quad \alpha \in (0, 32\pi^2).
$$

\vspace{0.5 cm}

Our main result is the following

\begin{theorem}\label{T1}
	Suppose that $(f_1)-(f_5)$ and $(V_1)-(V_3)$ hold. Then, for each non-empty subset
	$\Gamma \subset \{1,\cdots,k\}$, there exists $\lambda^*>0$ such that problem $(P_{\lambda})$ has a solution $u_{\lambda},$ provided that  $\lambda \ge \lambda^*.$ Moreover, the family $\{u_{\lambda}\}_{\lambda \ge \lambda^*}$ has the following property: for any sequence $\lambda_n \to \infty$, we can extract a
	subsequence $\lambda_{n_i}$ such that $u_{\lambda_{n_i}}$ converges
	strongly in $H^2(\R^4)$ to a function $u$ which satisfies $u(x)=0$
	for $x\notin \Omega_{\Gamma}=\cup_{j \in \Gamma}\Omega_j$, and the restriction $u|_{\Omega_j}$
	is a solution with least energy of

	\begin{equation}\label{4}
	\left\{ \begin{array}{c}
	\Delta^2 u  +u  = f(u), \quad \mbox{in} \quad \Omega_j \\
	u=\frac{\partial u}{\partial \eta}=0,\quad  \mbox{on} \quad \partial\Omega_j  .
	\end{array}
	\right.
	\end{equation}
\end{theorem}


\section{Preliminary remarks}

In this section we fix some notations and introduce some preliminary results that will be fundamental in the development of this work.
 
Hereafter, if $\Theta \subset \mathbb{R}^{4}$ is a mensurable set, we denote $E_\lambda(\Theta)$ the space 
$$
E_{\lambda}(\Theta)=\left\{ u\in H^{2}(\Theta);\ \int_{\Theta}V(x)\left|u\right|^2 dx <+\infty\right\}.
$$
endowed with the the inner product
$$
(u,v)_{\lambda,\Theta}= \int_{\Theta} \left[\Delta u\Delta v+(\lambda V(x)+1)uv\right]dx.
$$
The norm associated with this inner product will be denoted by $\|\cdot\|_{\lambda,\Theta}$.

The following estimate involving $f$, which is a consequence of our assumptions $(f_1)$ and $(f_2)$, is a key point in this work:

\noindent{\bf Main estimate.} Fixed $q\geq 0,~\eta>0$ and $\tau>1$, there exists $C>0$ such that
\begin{equation}\label{me}
|f(u)|\leq \eta|u|+C|u|^q(e^{32\pi^2\tau u^2}-1),~~\mbox{for all}~~u\in\mathbb{R}.
\end{equation}

In what follows, we establish an important inequality used in this paper. Its proof can be found in F. Sani \cite{S}.
\begin{lemma}\label{lm1}
Let $\alpha>0$ and $q\geq 2$. If $M>0$ and $\alpha M^2<32\pi^2$ then there exists a constant $C=C(\alpha,q,M)>0$ such that
$$
\int_{\mathbb{R}^4}(e^{\alpha u^2}-1)|u|^qdx\leq C\|u\|_{H^2}^q.
$$
holds for any $u\in H^2(\mathbb{R}^4)$ with $\|u\|_{H^2}\leq M$.
\end{lemma}

This inequality together with the Main estimate (\ref{me}) implies that the energy functional $I_{\lambda}:E_{\lambda}\rightarrow \mathbb{R}$, associated to $(P_{\lambda})$  given by
$$
I_{\lambda}(u)=\frac{1}{2}\int_{\mathbb{R}^4}\left[\left|\Delta u\right|^2+(\lambda V(x)+1)\left|u\right|^2\right]dx-\int_{\mathbb{R}^4}F(u)dx,
$$ is well defined. Furthermore, using standard arguments, it is possible to prove that $I_\lambda$ is of class $C^1$ with
$$
I_\lambda'(u)v=\int_{\mathbb{R}^4}(\Delta u\Delta v+(\lambda V(x)+1)uv)dx+\int_{\mathbb{R}^4}f(u)vdx ~~\mbox{for all}~~ u,v\in E_\lambda,
$$
and its critical points are solutions of our problem $(P_\lambda)$. 

To finish this section, we state the following consequence of the Adams-type inequality.

\begin{corollary}\label{cc1}
Let $\{u_n\}$ be a sequence in $H^2(\mathbb{R}^4)$ with $\displaystyle\limsup\{\|u_n\|^2_{H^2}\}\leq m<1$. For $\tau,~q>1$ satisfying $\tau qm<1$, there exists $C=C(\tau, q, m)>0$ such that $b_\tau(u_n):=(e^{32\pi^2\tau u_n^2}-1)$ belongs to $L^q(\mathbb{R}^4)$ and
$$
\|b_\tau(u_n)\|_{L^q(\mathbb{R}^4)}\leq C.
$$

\end{corollary}
\begin{proof}
Note that there exists $C>0$ such that
$$
|b_\tau(u_n)|^q\leq C\left(e^{32\pi^2(\sqrt{\tau q}u_n)^2}-1\right)
$$
and since $\|\sqrt{\tau q}u_n\|_{H^2}^2\leq\tau qm<1$, from Adams-type inequality $(\ref{X2})$
$$
\sup_{n}\{\|b_\tau(u_n)\|_{L^q(\mathbb{R}^4)}\}<\infty
$$
and the proof is completed.
\end{proof}

In the next, we define the constant $$S:=\sum_{j=1}^kc_j,$$ which  plays an important role in the study of Palais-Smale sequences, where $c_j$ is the minimax level of Mountain Pass Theorem related to the functional $I_j: H_0^2(\Omega_j)\to\mathbb{R}$ given by
$$
I_j(u)=\dfrac{1}{2}\int_{\Omega_j}(|\Delta u|^2+u^2)dx-\int_{\Omega_j}F(u)dx.
$$

The following result presents an estimate from above for the constant $S$. 
\begin{lemma}\label{S}
	If $(f_1)-(f_5)$ hold, then $S\in(0,(\theta-2)/4\theta\overline{C}^2)$.
\end{lemma}
\begin{proof}
	In order to prove this inequality, for each $j\in\{1,\cdots, k\}$, let us fix a positive function $u_j\in H^2_0(\Omega_j)$ such that
	$$
	\gamma_j=\displaystyle\inf_{u\in H_0^2(\Omega_j)}\dfrac{\left(\int_{\Omega_j}(|\Delta u|^2+u^2)dx\right)^{1/2}}{\left(\int_{\Omega_j}|u|^pdx\right)^{1/p}}=\dfrac{\left(\int_{\Omega_j}(|\Delta u_j|^2+u_j^2)dx\right)^{1/2}}{\left(\int_{\Omega_j}|u_j|^pdx\right)^{1/p}}.
	$$
	It follows from the minimax characterizaton of $c_j$ and hypothesis $(f_5)$ that
	$$
	\begin{array}{rcl}
	c_j&\leq &\max_{t\geq 0} I_j(tu_j)\\[0.5cm]
	&\leq& \max_{t\geq 0}\left[\dfrac{t^2}{2}\int_{\Omega_j}(|\Delta u_j|^2+|u_j|^2)dx-\dfrac{t^pC_p}{p}\int_{\Omega_j}|u_j|
	^pdx\right]\\[0.5cm]
	&= &\dfrac{p-2}{4p}\dfrac{\gamma_j^{2p/(p-2)}}{C_p^{2/(p-2)}},
	\end{array}
	$$
	hence
	$$
	S=\sum_{j=1}^kc_j\leq k\dfrac{p-2}{4p}\dfrac{S_p^{2p/(p-2)}}{C_p^{2/(p-2)}}.
	$$
	On the other hand, by $(f_5)$
	$$
	k\dfrac{p-2}{4p}\dfrac{S_p^{2p/(p-2)}}{C_p^{2/(p-2)}}<\dfrac{\theta- 2}{4\theta\overline{C}^2}.
	$$
\end{proof}

\section{The $(PS)_c$ Condition}

In this section, we study some results about the Palais-Smale sequences related to $I_\lambda$, that is, of sequences $\{u_n\}\subset E_\lambda$ verifying
$$
I_\lambda(u_n)\to c~~~\mbox{and}~~~I_\lambda'(u_n)\to 0,
$$
for some $c\in\mathbb{R}$ (shortly $\{u_n\}$ is a $(PS)_c$ sequence).
\begin{lemma}\label{l1}
	Let $\{u_n\} \subset E_{\lambda}$ be a $(PS)_c$ sequence  for
	$I_{\lambda},$ then $\{u_n\}$ is bounded. Furthermore, 
$$\limsup_{n\to\infty}\|u_n\|_\lambda^2\leq\dfrac{2\theta c}{\theta-2}.$$
\end{lemma}

\begin{proof}
	Since $\{u_n\}$ is a $(PS)_c$ sequence,
	$$
	I_{\lambda}(u_n)\rightarrow c\ \mbox{and}\ I'_{\lambda}(u_n) \rightarrow 0.
	$$
	Thereby, for $n$ large enough,
	\begin{equation}\label{5}
	I_{\lambda}(u_n)-\frac{1}{\theta}I'_{\lambda}(u_n)u_n \leq c+o_n(1)+\epsilon_n\left|\left|u_n\right|\right|_{\lambda},
	\end{equation}
where $\epsilon_n\to 0$.	On the other hand,
	$$
	I_{\lambda}(u_n)-\frac{1}{\theta}I'_{\lambda}(u_n)u_n=\left( \frac{1}{2}-\frac{1}{\theta}\right)\|u_n\|^2_{\lambda}+\int_{\mathbb{R}^4}\left[ \frac{1}{\theta}f(u_n)u_n-F(u_n)\right]\,dx.
	$$
	Then, by $(f_3)$, 
	\begin{equation}\label{6}I_{\lambda}(u_n)-\frac{1}{\theta}I'_{\lambda}(u_n)u_n\geq \left( \frac{1}{2}-\frac{1}{\theta}\right)\|u_n\|^2_{\lambda}.
	\end{equation}
	Gathering $ (\ref{5}) $ and $ (\ref{6}) $, we get 
	$$
	\left( \frac{1}{2}-\frac{1}{\theta}\right)\|u_n\|^2 _{\lambda} \leq c+o_n(1)+\epsilon_n\|u_n\|_{\lambda}.
	$$
	Therefore, $\{u_n\}$ is bounded and
	\begin{equation}\label{7}
\limsup_{n\to\infty}\|u_n\|_\lambda^2\leq\dfrac{2\theta c}{\theta-2}.
	\end{equation}
\end{proof}

An immediate consequence of the previous lemma is the following result.
\begin{corollary}\label{c1}
	Let $\{u_n\} \subset E_{\lambda}$ be a $(PS)_0$ sequence for
	$I_{\lambda}.$ Then, $u_n\rightarrow 0$ in $E_\lambda$.
\end{corollary}

\begin{lemma}\label{l2}
	Let $\{ u_n\}$ be a $(PS)_c$ sequence for
	$I_{\lambda}$ with $c\in (0,S]$. If $u_n \rightharpoonup u$ in $E_{\lambda},$ then
	\begin{eqnarray}
	I_{\lambda}(v_n)-I_{\lambda}(u_n)+I_{\lambda}(u) &=& o_n(1) \nonumber\\
	I'_{\lambda}(v_n)-I'_{\lambda}(u_n)+I' _{\lambda}(u) &= & o_n(1)\nonumber,
	\end{eqnarray}
	where $v_n:=u_n-u.$ Furthermore, $\{v_n\}$ is a 
	$(PS)_{c-I_{\lambda}(u)}$ sequence.
\end{lemma}
\begin{proof}
	Firstly, note that
	\begin{align*}
	I_{\lambda}(v_n)-I_{\lambda}(u_n)+I_{\lambda}(u) &=\frac{1}{2}\left( \|v_n\|^2_{\lambda}-\|u_n\|^2_{\lambda}+\|u\|^2_{\lambda}\right)&\\
	&\ \ -\int_{\mathbb{R}^4}\left( F(v_n)-F(u_n)+F(u)\right)dx &\\
	&=o_n(1)-\int_{B_R(0)}\left( F(v_n)-F(u_n)+F(u)\right)dx&\\
	&\ \ -\int_{\mathbb{R}^4 \setminus B_R(0)}\left( F(v_n)-F(u_n)+F(u)\right)dx ,&
	\end{align*}
	where $ R> $ 0 will be fixed later on. Since, $u_n \rightharpoonup u$ in $E_{\lambda}$, we have
	\begin{center}
		\begin{itemize}
			\item	$u_n \rightarrow u,\ \mbox{in}\ L^p(B_R(0)),\ \mbox{for}\ p \ge 1;$
			\item $u_n(x) \rightarrow u(x),\ a.e. \ \mbox{in}\ \mathbb{R}^4.$
		\end{itemize}
	\end{center}
	Moreover, there are $h_1 \in L^2(B_R(0))$ and $ h_2 \in L^q(B_R(0))$ such that 
	$$
	\left|u_n(x)\right|\leq h_1(x), h_2(x) \quad \mbox{a.e. in } \quad \mathbb{R}^{4}.
	$$
	By Lebesgue's Theorem,
	\begin{equation}\label{8}
	\int_{B_R(0)} \left| F(v_n)-F(u_n)+F(u)\right|\,dx \rightarrow 0.
	\end{equation}
	
	On the other hand, using the Main Estimate on $f$ and The Mean Value Theorem, there is $t_n\in[0,1]$ such that $\theta_n:=t_nv_n+(1-t_n)u_n$ satisfies
	$$
\begin{array}{rcl}
	\left|F\left(v_n\right)-F\left(u_n\right)\right| &\leq &\eta|f(\theta_n)||u|\\[0.4cm]
          &\leq & \eta |\theta_n||u|+C|u|b_\tau(\theta_n)\\[0.4cm]
         &\leq & \eta (|u_n|+|u|)|u|+C|u|b_\tau(\theta_n).
\end{array}
$$
Then, by H\"older inequalities
$$
\begin{array}{rcl}
	\int_{\mathbb{R}^4 \setminus B_R(0)}\left|F\left(v_n\right)-F\left(u_n\right)\right| &\leq &\eta(\|u_n\|_{L^2(\mathbb{R}^4\setminus B_R(0))}+\|u\|_{L^2(\mathbb{R}^4\setminus B_R(0))})\|u\|_{L^2(\mathbb{R}^4\setminus B_R(0))}\\[0.4cm]
          &&+ C\|u\|_{L^{q'}(\mathbb{R}^4\setminus B_R(0))}\|b_\tau(\theta_n)\|_{L^{q}(\mathbb{R}^4)},\\
\end{array}
$$
where $q',q>1$ satisfies $1/q'+1/q=1$. 

Since
\begin{equation}
	\|\theta_n\|_{H^2}\le \|u_n\|_{H^2}+\|u\|_{H^2},
\end{equation}
and $u_n \rightharpoonup u$, it follows that
\begin{equation}\label{e0}
\limsup_{n\to \infty}\|\theta_n\|_{H^2}\le 2\limsup_{n\to \infty}\|u_n\|_{H^2}
\end{equation}
from Lemmas \ref{l1}, \ref{S} and by embedding (\ref{imb}) follow that
\begin{equation}\label{e1}
	\limsup_{n\to \infty}\|u_n\|_{H^2}\le m <1/2,
\end{equation} where $m=\frac{2\theta S\overline{C}}{\theta-2}.$ 
Combining (\ref{e0}) and (\ref{e1}) we obtain
 $$\limsup\|\theta_n\|_{H^2}\leq m'< 1.$$
  Consindering $q>1$ suficiently close to $1$ such that $\tau qm'<1$, by Corollary $\ref{cc1}$, we get $\|b_\tau(\theta_n)\|_{L^{q}(\mathbb{R}^4)}\leq C$. Then, the above estimate combined with the boundedness of $\{u_n\}$ and  Sobolev embeddings gives  
	\begin{align*}
		\int_{\mathbb{R}^4 \setminus B_R(0)}\left|F \left(v_n\right)- F\left(u_n\right) \right|dx \leq &\,\,\, \eta C_1 \left( \|u\|_{L^2(\mathbb{R}^4 \setminus
			B_R(0))}+\|u\|_{L^2(\mathbb{R}^4 \setminus
			B_R(0))}^2\right)&\\
		& +C\|u\|_{L^{q'}(\mathbb{R}^4 \setminus
			B_R(0))}.\\
		\end{align*}
Now, we can fix $R>0$ large enough verifying 
	 $$ 
	\int_{\mathbb{R}^4
		\setminus B_R(0)}\left|F \left(v_n\right)-
	F\left(u_n\right) \right|dx \leq \epsilon.
	$$ 
	Using again the Main Estimate on $f$ we obtain
	$$
	\int_{\mathbb{R}^4\setminus B_R(0)}\left|
	F\left(u\right) \right|dx \leq \eta
	\left|\left|u\right|\right|_{L^2(\mathbb{R}^4 \setminus
		B_R(0))}^2+C\|u\|_{L^q(\mathbb{R}^4 \setminus
		B_R(0))}^q.
	$$ 
	Then, increasing $R$ if necessary, we can assume that
	$$
	\int_{\mathbb{R}^4 \setminus B_R(0)}\left| F\left(u\right)
	\right|dx \leq \epsilon.
	$$ 
	Hence,
	$$ 
	\int_{\mathbb{R}^4
		\setminus B_R(0)}\left|F \left(v_n\right)-
	F\left(u_n\right) +F\left(u\right)\right|dx \leq \epsilon, \quad \forall n \in \mathbb{N}.
	$$
	By arbitrariness of $ \epsilon, $ it follows that
	\begin{equation}\label{9}
	\limsup_{n \rightarrow +\infty} \int_{\mathbb{R}^4
		\setminus B_R(0)}\left|F \left(v_n\right)-
	F\left(u_n\right) +F\left(u\right)\right|dx=0
	\end{equation}
	From (\ref{8}) and (\ref{9}), we get the first of the identities. The second one follows exploring the same type of arguments and the growth of $f'$.
\end{proof}

\begin{lemma}\label{l3}
	Let $\{u_n\}$ be a $(PS)_c$ sequence for $I_{\lambda}$ with $c \in (0,S]$. Then, there exists $c_*>0$, independent of $\lambda,$  such that  $c \in [c_*,S] $.
\end{lemma}

\begin{proof}
	The Lemma \ref{lm1} together with the Sobolev embedding gives
	$$
	I'_{\lambda}(u)u \geq \frac{1}{2}\left|\left|u \right|\right|_{\lambda}^2-K\left|\left|u \right|\right|_{\lambda}^q,
	$$
	for some constant $K>0$.  Thus, there exists $\delta>0$ such that
	\begin{equation}\label{11}
	I'_{\lambda}(u)u \geq
	\frac{1}{4}\left|\left|u\right|\right|_{\lambda}^2,\ \mbox{for}\
	\left|\left|u\right|\right|_{\lambda} \leq\delta.
	\end{equation}
	Consider $c_*= \delta^2\frac{\theta -2}{2\theta}$ and 
	$c\in(0,c_*)$. From Lemma $\ref{l1}$, 
	$$
	\limsup_{n \rightarrow + \infty} \left|\left|u_n\right|\right|_{\lambda}^2 < \delta^2,
	$$
	implying that for $n $ large enough,
	\begin{equation}\label{12}
	\left|\left|u_n\right|\right|_{\lambda}\leq \delta.
	\end{equation}
	Hence, (\ref{11}) and (\ref{12}) combine to give 
	$$
	I'_{\lambda}(u_n)u_n \geq
	\frac{1}{4}\left|\left|u_n\right|\right|_{\lambda}^2,
	$$ 
	leading to	
	$$
	\left|\left|u_n\right|\right|_{\lambda}^2 \rightarrow 0,
	$$
	and so, $I_{\lambda}(u_n) \rightarrow I_{\lambda}(0)=0$. This contradicts the hypothesis that  $ \{u_n \} $ is a $ (PS) _c $ sequence, with
	$c>0$. Therefore, $c \geq c_*.$
\end{proof}

\begin{lemma}\label{l4}
	Let $\{u_n\}$ be a $(PS)_c$  sequence for $I_{\lambda}$, with $c \in [0,S]$. Then,
	there exists $\delta_0 >0$ independent
	of $\lambda,$ such that $$\liminf_{n
		\rightarrow + \infty}
	\left|\left|u_n\right|\right|_{L^{q}(\mathbb{R}^N)}^q \geq
	\delta_0c.$$
\end{lemma}

\begin{proof}
	By $(f_1)$ and $(f_2),$ given $\eta>0$, there is $C_\eta>0$ such that
	$$
	\frac{1}{2}f(t)t-F(t) \leq \eta\left|t\right|^2+C_{\eta}\left|t\right|^{q}b_{\tau}(t),\ \forall t \in \mathbb{R}.
	$$
	Then, from Corollary \ref{cc1} and arguing as in the proof of Lemma \ref{l2}
	\begin{equation}\label{13}
	c \leq \liminf_{n \rightarrow
		+\infty}\left(\eta\left|\left|u_n\right|\right|_{\lambda}^2+C_{\eta}\left|\left|u_n\right|\right|_{L^q(\mathbb{R}^4)}^{q}\right).
	\end{equation}
	On the other hand,  by $(f_3)$,
	\begin{equation}\label{14}
	I_{\lambda}(u_n)-\frac{1}{\theta}I'_{\lambda}(u_n)u_n\geq
	\left(\frac{1}{2}-\frac{1}{\theta}\right)\left|\left|u_n\right|\right|_{\lambda}^2.
	\end{equation}
	Combining $(\ref{13})$ with $(\ref{14})$, we get  
	$$
	c \leq \frac{2
		\eta c \theta}{\theta-2}+C_{\eta}\liminf_{n \rightarrow
		+\infty}\left|\left|u_n\right|\right|_{L^q(\mathbb{R}^4)}^{q}.
	$$
	Thereby, for $\eta$ small enough,
	$$\liminf_{n \rightarrow
		+\infty}\left|\left|u_n\right|\right|_{L^q(\mathbb{R}^4)}^{q} \geq
	\frac{c}{C_\eta}\left(1-\frac{2\eta
		\theta}{\theta-2}\right)>0.
		$$
Now, the lemma follows fixing
	$$
	\delta_0=\frac{1}{C_\eta}\left(1-\frac{2\eta \theta}{\theta-2}\right).
	$$ 
\mbox{\hspace{13 cm}} \end{proof}

The following result is a consequence of hypothesis $(V_3)$ and its proof can be found in \cite{AN,B&W1,B&W2}.
\begin{lemma}\label{l5}
	Let $\{u_n\}$ be a $(PS)_c$ sequence for $I_{\lambda}$ with $c \in [0,S]$. Given $\epsilon >0$, there exist constants $\Lambda,R>0$ such that
	$$
	\limsup_{n \rightarrow +\infty}\left|\left|u_n\right|\right|_{L^q(\mathbb{R}^4 \setminus B_R(0))}^{q} \leq \epsilon,~~ \mbox{for all}~~ \lambda \geq \Lambda.
	$$
\end{lemma}
In the following result, we will show that the functional $I_{\lambda}$ verifies the $(PS)_c$ condition to $c\in[0,S]$ and $\lambda$ large enough.
\begin{proposition}\label{p1}
	Let $c\in[0,S]$. Then, there exists $\Lambda>0$ such that $I_{\lambda}$
	verifies the  $(PS)_c$ condition, for all  $\lambda\geq \Lambda$.
\end{proposition}

\begin{proof}
	Let $\{u_n\}$ be a $(PS)_c$ sequence. By Lemma \ref{l1}, $\{u_n\}$ is bounded and consequently, passing to a subsequence if necessary,
	$$\left\{ \begin{array}{l}
	u_n\rightharpoonup u~~ \mbox{in}~~ E_{\lambda};\\
	u_n(x)\rightarrow u(x) ~~ a.e.~~ \mbox{in}~~ \mathbb{R}^4;\\
	u_n\rightarrow u~~ \mbox{in}\ L^{s}_{loc}(\mathbb{R}^4),~~ s\geq 1.
	\end{array}
	\right.$$
	Then, $I_{\lambda}'(u)=0$  and $I_{\lambda}(u)\geq 0$, because
$$
	I_{\lambda}(u)=I_{\lambda}(u)-\frac{1}{\theta}I'_{\lambda}(u)u\geq \left( \frac{1}{2}-\frac{1}{\theta}\right)\|u\|^{2}_{\lambda} \geq 0.
	$$
	Taking $v_n=u_n-u,$ we have by Lemma \ref{l2}
	that $\{v_n\}$ is a $(PS)_{d}$ sequence, with $d=c-I_{\lambda}(u)$. Furthermore, 
	$$
	0\leq d=c-I_{\lambda}(u)\leq c \leq S.
	$$
	We claim that $d=0$. Indeed, otherwise $d>0$. Thereby, by Lemma \ref{l3} and Lemma \ref{l4}, $d \geq c_*$
	and 
	\begin{equation}\label{18}
	\liminf_{n \rightarrow +\infty} \| v_n\|_{L^q(\mathbb{R}^4)}^q \geq
	\delta_0c_*>0.
	\end{equation}
	Applying the Lemma \ref{l5} with $\epsilon=\frac{\delta_0c_*}{2}>0,$
	there exist $\Lambda, R>0$ such that 
	\begin{equation}\label{19}
	\limsup_{n \rightarrow +\infty} \| v_n\|_{L^q(\mathbb{R}^4)\setminus
		B_R(0)}^q \leq \frac{\delta_0c_*}{2}, \quad \mbox{for} \quad \lambda \geq \Lambda.
	\end{equation}
	Combining $(\ref{18})$ with $(\ref{19})$, we obtain
	$$\liminf_{n \rightarrow
		+\infty} \left|\left| v_n\right|\right|_{L^q(B_R(0))}^q \geq
	\frac{\delta_0c_*}{2}>0,
	$$ 
	which is an absurd, because as $v_n \rightharpoonup 0$ in $E_\lambda$, the compact embedding   
	\,\,\,\,\,\,\,\,\,\,\,\,\,\,\,\,\,\,\,\,\,\,\,\,\,$E_{\lambda}\hookrightarrow L^q(B_R(0))$ gives 
	$$
	\liminf_{n
		\rightarrow +\infty} \| v_n\|_{L^q(B_R(0))}^q=0.
	$$ 
	Therefore $d=0$  and $\{v_n\}$ is a $(PS)_0$ sequence. Then, by Corollary
	\ref{c1}, $v_n \rightarrow 0$ in $E_\lambda$, or equivalently, $u_n \rightarrow u$ in $E_\lambda$, showing that for $\lambda$ large enough, $I_{\lambda}$
	satisfies the $(PS)_c$ condition for all $c \in [0,S].$
\end{proof}

\section{The $(PS)_{\infty}$ Condition}

In this section, we will study  the behavior of a  $(PS)_{\infty}$ sequence, that is, a sequence $\{u_n\} \subset H^2(\mathbb{R}^4)$ satisfying:
\begin{align*}
&u_n \in E_{\lambda_n}\ \mbox{and}\ \lambda_{n} \rightarrow +\infty;&\\
&I_{\lambda_n}(u_n)\rightarrow c,\quad \mbox{for some} \quad c\in [0,S];&\\
&\|I_{\lambda_n}'(u_n)\|_{E'_{\lambda_n}} \rightarrow 0.&
\end{align*}

 In the sequel, let us fix a bouded open subset $\Omega_j'$ with smooth boundary such that
\begin{enumerate}
\item[$(i)$] $\overline{\Omega_j}\subset\Omega_j'$;
\item[$(ii)$] $\overline{\Omega_j'}\cap \overline{\Omega_l'}=\emptyset$, for all $j\neq l$,
\end{enumerate}
and for $\Gamma\subset\{1,\cdots,k\},~\Gamma\neq\emptyset$, let us define
$$
\Omega_\Gamma=\bigcup_{j\in\Gamma}\Omega_j~~\mbox{and}~~\Omega_\Gamma'=\bigcup_{j\in\Gamma}\Omega_j'.
$$

\begin{proposition}\label{p2}
	Let $\{u_n\}$ be a $(PS)_{\infty}$ sequence for $I_{\lambda}$.
	Then, there is a subsequence of $\{u_n\}$, still denoted by itself,
	and $u \in H^2(\mathbb{R}^4)$ such that
$$
u_n \rightharpoonup u\ \mbox{in}\  H^2(\mathbb{R}^4).
$$
Moreover,
	\begin{description}
		\item[i)]$u\equiv 0$ in $\mathbb{R}^4 \setminus \Omega_{\Gamma}$ and $u$ is a solution of
		\begin{equation}\label{20}
		\left\{ \begin{array}{c}
		\Delta^2 u  +u  =  f(u),\mbox{in}\ \Omega_j, \ \\
		u=\dfrac{\partial u}{\partial \eta} =0,\ \mbox{on}\  \partial\Omega_j,
		\end{array}
		\right.
		\end{equation}
		for all $j \in \Gamma;$
		\item[ii)]$\left|\left| u_n-u\right|\right|^{2}_{\lambda_{n}} \rightarrow 0.$ 
		\item[iii)]$\left\lbrace u_n\right\rbrace $ also satisfies
		\begin{align*}
		&\lambda_n \int_{\mathbb{R}^4}V(x)\left|u_n\right|^2dx \rightarrow 0,\ n \rightarrow +\infty&\\
		&\left|\left|u_n\right|\right|^2_{\lambda_n,\mathbb{R}^4\setminus \Omega_{\Gamma}}\rightarrow 0 &\\
		&\left|\left|u_n\right|\right|^2_{\lambda_n,\Omega'_j}\rightarrow \int_{\Omega_j}\left[ \left|\Delta u\right|^2+\left|u\right|^2\right]dx,\ \forall j\in \Gamma.&
		\end{align*}
	\end{description}
\end{proposition}

\begin{proof}
	In what follows, we fix $c \in [0,S]$ verifying
	$$
	I_{\lambda_n}(u_n)\rightarrow c\ \mbox{and}\ \|I'_{\lambda_n}(u_n)\|_{E'_{\lambda_n}} \rightarrow 0.
	$$
	Then, there exists $n_0 \in \mathbb{N}$ such that, 
	$$
	I_{\lambda_n}(u_n)-\frac{1}{\theta}I'_{\lambda_n}(u_n)u_n \leq c+1+\left|\left|u_n\right|\right|_{\lambda_n}, \quad \forall n \ge n_0.
	$$
	On the other hand, from $(f_3)$, 
	$$
	I_{\lambda_n}(u_n)-\frac{1}{\theta}I'_{\lambda_n}(u_n)u_n \geq \left(\frac{1}{2}-\frac{1}{\theta}\right)\left|\left|u_n\right|\right|^2_{\lambda_n},\ \forall n \in \mathbb{N}.
	$$
	So, for $n\ge n_0$,
	$$
	\left(\frac{1}{2}-\frac{1}{\theta}\right)\left|\left|u_n\right|\right|^2_{\lambda_n} \leq c+1+\left|\left|u_n\right|\right|_{\lambda_n},
	$$
	implying that $\{\left|\left|u_n\right|\right|_{\lambda_n}\}$ is bounded in $\mathbb{R}$. As
	$$
	\left|\left|u_n\right|\right|_{\lambda_n} \geq \left|\left|u_n\right|\right|_{H^2(\mathbb{R}^4)}, \ \forall n \in \mathbb{N},
	$$
	$\{u_n\}$ is also bounded in $H^2(\mathbb{R}^4)$, and so, there exists a subsequence of $\{u_n\}$, still denoted by itself, and $u \in H^{2}(\mathbb{R}^{4})$ such that
	$$
	u_n \rightharpoonup u\ \mbox{in}\ H^2(\mathbb{R}^4).
	$$
	To show $(i)$, we fix for each $m \in \mathbb{N}^*,$
	the set
	$$
	C_m=\left\{x\in \mathbb{R}^4/ V(x) > \frac{1}{m}\right\}.
	$$
	Hence
	$$\mathbb{R}^N\setminus \overline{\Omega}=\bigcup_{m=1}^{+\infty}C_m.$$
	Note that,
	\begin{align*}
	\int_{C_m}\left|u_n\right|^2dx&=\int_{C_m}\frac{\lambda_n V(x)+1}{\lambda_n V(x)+1}\left|u_n\right|^2dx&\\
	&\leq \frac{1}{\frac{\lambda_n}{m}+1}\left|\left|u_n\right|\right|_{\lambda_n}^2&\\
	&\leq \frac{mM}{{\lambda_n}+m},&
	\end{align*}
	where $M=\sup_{n \in \mathbb{N}}\|u_n\|_{\lambda_n}^2.$ By Fatou's Lemma
	\begin{align*}
	\int_{C_m}\left|u\right|^2dx&\leq \liminf_{n \rightarrow +\infty}\int_{C_m}\left|u_n\right|^2dx&\\
	&\leq \liminf_{n \rightarrow +\infty}\frac{mM}{{\lambda_n}+m}=0.&
	\end{align*}
	Therefore, $u=0$ almost everywhere in $C_m$, and consequently, $u=0$ almost everywhere in $\mathbb{R}^4\setminus \overline{\Omega}.$
	Besides, fixing $\varphi \in C_{0}^{\infty}(\mathbb{R}^4\setminus \overline{\Omega})$, we have
	$$
	\int_{\mathbb{R}^4\setminus \overline{\Omega}}\nabla u(x)\varphi(x)dx=-\int_{\mathbb{R}^4\setminus \overline{\Omega}} u(x)\nabla\varphi(x)dx=0,
	$$
	from where it follows that
	$$
	\nabla u(x)=0,\ a.e. \ \mbox{in}\ \mathbb{R}^4\setminus \overline{\Omega}.
	$$
	Since $\partial \Omega$ is smooth ,  $u \in H^2(\mathbb{R}^4\setminus \overline{\Omega})$ and $\nabla u \in H^1(\mathbb{R}^4\setminus \overline{\Omega}),$  by Trace Theorem, there are constants $K_1,K_2>0$ satisfying 
	$$
	\left|\left|u\right|\right|_{L^2(\partial \Omega)} \leq K_1\left|\left|u\right|\right|_{H^2(\mathbb{R}^4\setminus \overline{\Omega})}=0,
	$$
	and
	$$
	\left|\left|\nabla u\right|\right|_{L^2(\partial \Omega)} \leq K_2 \left|\left|\nabla u\right|\right|_{H^1(\mathbb{R}^4\setminus \overline{\Omega})}=0,
	$$ 
loading to $u \in H_0^2(\Omega).$ To complete the proof of $i)$, consider a test function $\varphi \in C_{0}^{\infty}(\Omega)$, and note that
	\begin{equation}\label{21}
	I'_{\lambda_n}(u_n)\varphi=\int_{\Omega}\left[\Delta u_n \Delta \phi + u_n \varphi\right]dx-\int_{\Omega}f(u_n)\varphi dx.
	\end{equation} 
Since $\left\lbrace u_n\right\rbrace $ is a $(PS)_{\infty}$ sequence, we derive that
	\begin{equation}\label{22}
	I'_{\lambda_n}(u_n)\varphi \rightarrow 0.
	\end{equation}
Recalling that $u_n \rightharpoonup u$ in $H^2(\mathbb{R}^{4})$, we must have 
	\begin{equation}\label{23}
	\int_{\Omega}\left[\Delta u_n \Delta \varphi + u_n \varphi\right]dx \rightarrow \int_{\Omega}\left[\Delta u \Delta \varphi + u \varphi\right]dx
	\end{equation}
	and
	\begin{equation}\label{24}
	\int_{\Omega}f(u_n)\varphi dx \rightarrow \int_{\Omega}f(u)\varphi dx.
	\end{equation}
	Therefore, from (\ref{21})-(\ref{24}), 
	$$
	\int_{\Omega}\left[\Delta u \Delta \varphi +u \phi\right]dx=\int_{\Omega}f(u)\varphi dx,\ \forall \varphi \in C_0^{\infty}(\Omega).
	$$
	As $C_{0}^{\infty}(\Omega)$ is dense in $H_0^2(\Omega)$, the above equality gives
	$$
	\int_{\Omega}\left[\Delta u \Delta v +uv\right]dx=\int_{\Omega}f(u)v dx,\ \forall v \in H_0^2(\Omega),
	$$
	showing that $u$ is weak solution of the problem
	\begin{equation}\label{25}
	\left\{ \begin{array}{c}
	\Delta^2 u  +u  =  f(u),\, \mbox{in} \, \Omega_j, \ \\
	u=\dfrac{\partial u}{\partial \eta} =0,\ \mbox{on}\  \partial\Omega_j,
	\end{array}
	\right.
	\end{equation}
	For $ii)$, note that
	\begin{align} \label{26}
	\left|\left|u_n-u\right|\right|_{\lambda_n}^2&= \left|\left|u_n\right|\right|_{\lambda_n}^2+\left|\left|u\right|\right|_{\lambda_n}^2-2\int_{\mathbb{R}^4}\left[\Delta u_n\Delta u+ (\lambda_n V(x)+1)u_nu\right]dx.& 
	\end{align}
From $i)$, 
$$
	\|u\|_{\lambda_n}^2=\|u\|^2_{H_0^2(\Omega)},
$$
and so, 
	$$
	\int_{\mathbb{R}^4}\left[\Delta u_n\Delta u+ (\lambda_n V(x)+1)u_nu\right]dx=\|u\|_{H_0^2(\Omega)}+o_n(1).
	$$
	Thus, we can rewrite (\ref{26}) as
	\begin{align} \label{27}
	\left|\left|u_n-u\right|\right|_{\lambda_n}^2&= \left|\left|u_n\right|\right|_{\lambda_n}^2-\left|\left|u\right|\right|_{H_0^2(\Omega)}^2+o_n(1).& 
	\end{align}
Gathering the boundedness of $\{\|u_n\|_{\lambda_n}\}$ with the limit $ \|I'_{\lambda_n}(u_n)\|_{E'_{\lambda_n}}\rightarrow 0,$ we find the limit 
	$$
	I'_{\lambda_n}(u_n)u_n \rightarrow 0.
	$$
	Hence,
	\begin{equation} \label{28}
	\left|\left|u_n\right|\right|_{\lambda_n}^2= I'_{\lambda_n}(u_n)u_n+\int_{\mathbb{R}^4}f(u_n)u_ndx =\int_{\mathbb{R}^4}f(u_n)u_ndx+o_n(1).
	\end{equation}
	On the other hand,  we know that the limit  $	I'_{\lambda_n}(u_n)u \to 0 $ is equivalent to 
  $$
	\int_{\Omega}\left[ \Delta u_n \Delta u + u_nu\right]dx-\int_{\Omega}f(u_n)udx=o_n(1),
	$$
	which loads to 
	\begin{equation} \label{29}
	\int_{\mathbb{R}^4}\left[ \left|\Delta u\right|^2 +
	\left|u\right|^2\right]dx=\int_{\mathbb{R}^4}f(u)udx.
	\end{equation}
	Combining (\ref{27}) with (\ref{28}) and (\ref{29}), we see that 
	$$
	\left|\left|u_n-u\right|\right|_{\lambda_n}^2=\int_{\mathbb{R}^4}f(u_n)u_ndx-\int_{\mathbb{R}^4}f(u)udx+o_n(1).
	$$
	Using the Lebesgue's Theorem together with Adams-type inequality $(\ref{X2}),$ we get  
	$$
	\int_{\mathbb{R}^4}f(u_n)u_ndx\rightarrow \int_{\mathbb{R}^4}f(u)udx,
	$$
	finishing the proof of $ii)$. 
	The proof of $iii)$ follows from $ii)$ and the inequality below
	$$
	\lambda_n\int_{\mathbb{R}^4}V(x)\left|u_n\right|^2dx=\lambda_n\int_{\mathbb{R}^4}V(x)\left|u_n-u\right|^2dx
	\leq \|u_n-u\|_{\lambda_n}^2 .
	$$ \hspace{13.5cm}
\end{proof}

\section{The Existence of Multi-bump Solutions }

In this section, we denote by  $I_{\lambda,j}:H^2(\Omega'_j)  \rightarrow \mathbb{R}
$ the functional given by

$$
I_{\lambda,j}(u)=\frac{1}{2}\int_{\Omega'_j}\left[\left|\Delta u\right|^2+(\lambda V(x)+1)\left| u\right|^2\right]dx-\int_{\Omega'_j}F(u)dx.
$$

It follows from \cite{S} that $I_j$ and  $I_{\lambda,j}$, where $I_j $ is the functional defined in the section $2$, satisfy the mountain pass geometry. Hereafter, we denote by $c_j$ and
$c_{\lambda,j}$ the mountain pass levels related to the functionals  $I_j$ and $I_{\lambda,j}$. Since those functionals satisfy to the
Palais-Smale condition in the interval $[0,S]$ and $c_{\lambda,j}\le c_j \le S,\ \mbox{for all}\ j \in \{1,\cdots,k\}$, we conclude from Mountain Pass Theorem due to
Ambrosetti-Rabinowitz that there exist $w_j \in H_0^2(\Omega_j)$ and  $v_j
\in H^2_0(\Omega'_j)$ such that
$$
I_j(w_j)=c_j,\ I_{\lambda,j}(v_j)=c_{\lambda,j}\ \mbox{and}\ I'_j(w_j)=I'_{\lambda,j}(v_j)=0.
$$

In what follows, consider $\Gamma=\{1,2,\cdots,l\}$, with $l \leq k$, and $c_{\Gamma}=\sum_{j=1}^{l}c_j$. We fix $\epsilon>0$ and $\zeta>0$ such that   
$$I_j((1-\epsilon)w_j, I_j((1+\epsilon)w_j)< c_j-\zeta, \forall j \in \Gamma.
$$

Let us set $Q=(1-\epsilon,1+\epsilon)^l$ and define  $\gamma_0:\overline{Q} \rightarrow E_{\lambda}$ by
$$
\gamma_0(\overrightarrow{s})(x)=\sum_{j=1}^{l}s_jw_j(x),\ \forall \overrightarrow{s}=(s_1,\cdots,s_l)\in \overline{Q}.
$$
In the case of polynomial subcritical growth  $(N\ge 4),$ Alves and N\'obrega in \cite{AN} considered a cube $\tilde{Q}=(1/R^2,1)^l,$ where $R>0$ was chosen large, thus $|\tilde{Q}|$ was near to $1$. In our case, since we are working with critial exponential growth, we need to consider a cube  $Q=(1-\epsilon,1+\epsilon)^l$,  where $\epsilon>0$ will be taken small, which implies $|Q|$ is near to $0$.
   
In what follows, we denote by $\Sigma_{\lambda}$ the class of continuous path $\gamma \in C(\overline{Q},E_{\lambda}\setminus \{0\})$ satisfying the following conditions: 
$$
\gamma=\gamma_0\ \mbox{on}\ \partial(\overline{Q}) \leqno{(a)}
$$
and 
$$
I_{\lambda,\mathbb{R}^4\setminus \Omega'_{\Gamma}}(\gamma(\overrightarrow{s}))\ge 0, \leqno{(b)}
$$
where $I_{\lambda,\mathbb{R}^4\setminus \Omega'_{\Gamma}}: H^{2}(\mathbb{R}^4\setminus \Omega'_{\Gamma}) \to \mathbb{R}$ is the functional defined  by
$$
I_{\lambda,\mathbb{R}^4\setminus \Omega'_{\Gamma}}(u)=\frac{1}{2}\int_{\mathbb{R}^4\setminus \Omega'_{\Gamma}}\left[\left|\Delta u\right|^2+(\lambda V(x)+1)\left| u\right|^2\right]dx-\int_{\mathbb{R}^4\setminus \Omega'_{\Gamma}}F(u)dx.
$$
Notice that $\Sigma_{\lambda} \neq \emptyset$, because $\gamma_0 \in \Sigma_{\lambda}$. 

Using the class $\Sigma_{\lambda}$, we define the following minimax level
$$
b_{\lambda,\Gamma}=\inf_{\gamma \in \Gamma_*}\max_{\overrightarrow{s} \in \overline{Q}}I_{\lambda}(\gamma(\overrightarrow{s})).
$$

\begin{lemma}\label{l6}
	For each $\gamma \in \Gamma_*,$ there is
	$\overrightarrow{t}\in\overline{Q}$ verifying 
	$$
	I'_{\lambda,j}(\gamma(\overrightarrow{t}))\gamma(\overrightarrow{t})=0,\ \mbox{for}\ j\in\Gamma.
	$$
\end{lemma}
\begin{proof}
	Given $\gamma \in \Sigma_{\lambda}$, consider the map $\widetilde{\gamma}:\overline{Q} \rightarrow \mathbb{R}^l$ defined by
	$$
	\widetilde{\gamma}(\overrightarrow{s})=\left(I'_{\lambda,1}(\gamma(\overrightarrow{s}))\gamma(\overrightarrow{s}), \cdots, I'_{\lambda,l}(\gamma(\overrightarrow{s}))\gamma(\overrightarrow{s})	\right).
	$$ 
	For $\overrightarrow{s} \in \partial(\overline{Q}),$ we know that
	$$
	\gamma(\overrightarrow{s})=\gamma_0(\overrightarrow{s}).
	$$
	Then,
	$
	I'_{\lambda,j}(\gamma_0(\overrightarrow{s}))(\gamma_0(\overrightarrow{s}))=0$
	which implies
	 $$s_j \not \in \{1-\epsilon,1+\epsilon\}, \forall j \in \Gamma.
	$$
	In fact, otherwise
	$$
	I'_{\lambda,j}(\gamma_0(\overrightarrow{s}))(\gamma_0(\overrightarrow{s}))=0
	$$
	for $s_{j}=(1-\epsilon)$ or $s_{j}=(1+\epsilon)$, that is, 
	$$
	I'_{j}((1-\epsilon)w_j)((1-\epsilon)w_j)=0 \quad \mbox{or} \quad I'_{j}((1+\epsilon)w_j)((1+\epsilon)w_j)=0
	$$
	implying  that 
	$$
	I_{j}((1-\epsilon)w_j)\geq c_j \quad \mbox{or} \quad I_{j}((1+\epsilon)w_j)\geq c_j,
	$$
	which contradicts the choice of $\epsilon$. Hence, 
	$$
	(0,0,\cdots,0) \not 	\in\widetilde{\gamma}(\partial(\overline{Q})).
	$$ 
	Then, by Topological Degree 
	$$
	deg(\widetilde{\gamma},Q,(0,0,\cdots,0))=(-1)^l\not= 0,
	$$
	from where it follows that there exists $\overrightarrow{t}\in Q$
	satisfying
$$
	I'_{\lambda,j}(\gamma(\overrightarrow{t}))(\gamma(\overrightarrow{t}))=0,\ \mbox{for}\ j \in \Gamma.
$$
\mbox{\hspace{13,4 cm}}\end{proof}

\begin{proposition}\label{p3} \mbox{}\\
\noindent $a) \,\, \sum_{j=1}^{l}c_{\lambda,j}\leq b_{\lambda, \Gamma} \leq c_{\Gamma},\, \forall \lambda \geq 1.$ \\
\noindent $b)$ \,\, For $\gamma \in \Gamma_*\ \mbox{and}\ \overrightarrow{s}\in \partial(\overline{Q})$, we have 
$$
I_{\lambda}(\gamma(\overrightarrow{s}))< c_{\Gamma},\, \forall \lambda \geq 1. 
$$

\end{proposition}

\begin{proof}
	\begin{description}
		\item[a)]Since $\gamma_0 \in \Sigma_{\lambda}$, 
		\begin{align*}
		b_{\lambda,\Gamma} &\leq \max_{\overrightarrow{s}\in \overline{Q}}I_{\lambda,j}(\gamma_0(\overrightarrow{s}))&\\
		&\leq \max_{\overrightarrow{s}\in \overline{Q}}I_{\lambda,j}(\sum_{i=1}^{l}s_iw_i(x))&\\
		&\leq \sum_{j=1}^{l}\max_{s_j \in [(1-\epsilon),(1+\epsilon)]}I_{j}(s_jw_j(x))&\\
		&\leq \sum_{j=1}^{l}c_j=c_{\Gamma.}&
		\end{align*}
		For each $\gamma \in \Sigma_{\lambda}$ and  $\overrightarrow{t}\in \overline{Q}$ as in Lemma \ref{l6}, we find 
		$$
		I_{\lambda,j}(\gamma(\overrightarrow{t})) \geq c_{\lambda,j}, \forall j \in \Gamma,
		$$
		where we have used the following characterization of $c_{\lambda,j}$
		$$
		c_{\lambda,j}=\inf\{I_{\lambda,j}(u);\ u \in E_{\lambda}\setminus \{0\};\,  I'_{\lambda,j}(u)u=0\}. 
		$$ 
		On the other hand, recalling that $I_{\lambda,\mathbb{R}^4\setminus \Omega'_{\Gamma}}(\gamma(\overrightarrow{s}))\geq 0$, we have
		$$
		I_{\lambda}(\gamma(\overrightarrow{s})) \geq
		\sum_{j=1}^{l}I_{\lambda,j}(\gamma(\overrightarrow{s})),
		$$ 
		and so,
		$$
		\max_{\overrightarrow{s}\in \overline{Q}}I_{\lambda}(\gamma(\overrightarrow{s})) \geq I_{\lambda}(\gamma(\overrightarrow{t})) \geq \sum^{l}_{j=1}c_{\lambda,j}.
		$$ 
		The last inequality combined with the definition of $b_{\lambda, \Gamma}$ gives  
		$$
		b_{\lambda, \Gamma}\geq \sum_{j=1}^{l}c_{\lambda, j},
		$$ 
		This completes the proof of $a).$
		\item[b)] As $\gamma(\overrightarrow{s})=\gamma_0(\overrightarrow{s})$ on $\partial(\overline{Q}),$ we derive that
		$$
		I_{\lambda}(\gamma_0(\overrightarrow{s}))=\sum_{j=1}^{l}I_j(s_jw_j), \forall \overrightarrow{s} \in \partial(\overline{Q}).
		$$
		Since 
		$$
		I_j(s_jw_j) \leq c_j, \quad \forall j \in \Gamma
		$$ 
		and there is $j_0 \in \Gamma$, such that $s_{j_0} \in \{(1-\epsilon),(1+\epsilon) \}$, we have  
		$$
		I_{\lambda}(\gamma_0(\overrightarrow{s})) < c_{\Gamma}.
		$$
		
	\end{description}
\mbox{\hspace{13 cm}} \end{proof}

\begin{corollary}
	
	$b_{\lambda, \Gamma} \rightarrow c_{\Gamma},$ when $\lambda
	\rightarrow +\infty.$
	
\end{corollary}

\begin{proof}
Using well known arguments, it is possible to prove that $c_{\lambda, j} \rightarrow c_j$ for each $j \in \Gamma$. Therefore, by Proposition $\ref{p3}$, $b_{\lambda,\Gamma} \rightarrow
	c_{\Gamma}$ when $\lambda \rightarrow +\infty.$
\end{proof}

\section{Proof of the Main Theorem}
Hereafter, we consider
$$
R=1+\sum_{j=1}^{l}\sqrt{\left(\frac{1}{2}-\frac{1}{\theta}\right)c_j},
$$
$$
\overline{B}_{R}(0)=\{ u \in E_{\lambda};\|u\|_{\lambda}\leq R+1\},
$$ 
and for small $\mu>0$, we define
$$
A_{\mu}^{\lambda}=\left\{ u \in \overline{B}_{R+1}; \left|\left|u\right|\right|_{\lambda,\mathbb{R}^N\setminus \Omega'_{\Gamma}}\leq \mu,\ I_{\lambda,\mathbb{R}^N\setminus 	\Omega'_{\Gamma}}(u)\geq 0\ \mbox{and}\ \left|I_{\lambda,j}(u)-c_j\right|\leq \mu, \forall j \in  \Gamma \right\},
$$
and
$$
I_{\lambda}^{c_{\Gamma}}=\left\{ u \in E_{\lambda}/ I_{\lambda}(u)\leq c_{\Gamma}\right\}.
$$
Note that $A_{\mu}^{\lambda}\cap I_{\lambda}^{c_{\Gamma}} \neq \emptyset,$  because $w= \sum_{j=1}^{l}w_j \in A_{\mu}^{\lambda}\cap I_{\lambda}^{c_{\Gamma}} .$

Another important estimate is that, for $\epsilon>0$ small enough,
\begin{equation}
	\|\gamma_0(\overrightarrow{s})\|^2_{\lambda}\le (1+\epsilon)^2\|\sum_{j=1}^{l}w_j\|^2_{\lambda}\le M=\frac{\theta c_{\Gamma}}{\theta-2}(1+\epsilon)^2<1.
\end{equation}
 Fixing
\begin{equation}\label{30}
0< \mu < \frac{1}{4}\min\{c_j;j \in  \Gamma\}
\end{equation}
we have the following uniform estimate from below for $\|I'_{\lambda}(u)\|$ in the set $\left(A_{2\mu}^{\lambda}\setminus
A_{\mu}^{\lambda}\right)\cap I_{\lambda}^{c_{\Gamma}}.$

\begin{proposition}\label{p4}
	Let $\mu >0$ satisfy $(\ref{30})$. Then, there exist $\sigma_0>0$ independent of $\lambda$ and
	$\Lambda_* \geq 1$ such that
	\begin{equation}\label{31}
	\ \|I'_{\lambda}(u)\| \geq \sigma_0\ \mbox{for}\ \lambda \geq
	\Lambda_*\ \mbox{and for all}\ u \in
	\left(A_{2\mu}^{\lambda}\setminus A_{\mu}^{\lambda}\right)\cap
	I_{\lambda}^{c_{\Gamma}}.
	\end{equation}
\end{proposition}

\begin{proof}
	Arguing by contradiction, suppose that there are $\lambda_n \rightarrow +\infty$  and $u_n \in E_{\lambda_n}$, with
	$$
	u_n \in \left(A_{2\mu}^{\lambda_n}\setminus A_{\mu}^{\lambda_n}\right)\cap I_{\lambda_n}^{c_{\Gamma}} \quad \mbox{and} \quad \|I'_{\lambda_n}(u_n)\| \rightarrow 0.
	$$
	Since $u_n \in A_{2\mu}^{\lambda_n}$, the sequence
	$\left\{\|u_n\|_{\lambda_n}\right\}$ is bounded. Consequently
	$\left\{I_{\lambda_n}(u_n)\right\}$ is also bounded. Then,  passing to a subsequence if necessary,
	$$
	I_{\lambda_n}(u_n) \rightarrow c \in (-\infty,c_{\Gamma}].
	$$
		By Proposition \ref{p2}, passing to a subsequence if necessary, $u_n
	\rightarrow u$ in $H^2(\mathbb{R}^4)$  and $u \in
	H_{0}^{2}(\Omega_{\Gamma})$ is a solution of the problem (\ref{20}). Moreover,
	\begin{align}
	&\lambda_n \int_{\mathbb{R}^4}V(x)\left|u_n\right|^2dx \rightarrow 0,& \label{32}\\
	&\left|\left|u_n\right|\right|^2_{\lambda_n,\mathbb{R}^4\setminus \Omega_{\Gamma}}\rightarrow 0 & \label{33}\\
	&\left|\left|u_n\right|\right|^2_{\lambda_n,\Omega'_j}\rightarrow \int_{\Omega_j}\left[ \left|\Delta u\right|^2+\left|u\right|^2\right]dx,\ \forall j\in \Gamma.& \label{34}
	\end{align}
	Since $c_{\Gamma}=\sum_{j=1}^{l}c_j$ and $c_j$ is the least energy level for $I_j$, one of the following cases occurs:
	\begin{description}
		\item[i)] $u\left|_{\Omega_j}\neq 0\right.,\ \forall j \in \Gamma,$ or
		\item[ii)]$u\left|_{\Omega_{j_0}}=0\right.,$ for some $j_0 \in \Gamma.$
	\end{description}
	If $i)$ happens, from $(\ref{32})-(\ref{34})$ 
	$$
	I_j(u)=c_j, \quad \forall j \in \Gamma.
	$$
	Hence $u_n \in A_{\mu}^{\lambda_n}$ for $n$ large enough, which is a contradiction. 
	
	If $ii)$ happens, from $(\ref{32})\ \mbox{and}\ (\ref{33})$  $$\left|I_{\lambda_{n},j_0}(u_n)-c_{j_0})\right| \rightarrow c_{j_0} \geq 4\mu,$$
	which contradicts the  hypothesis  $u_n \in A_{2 \mu}^{\lambda_n}\setminus A_{\mu}^{\lambda_n}$ for all $n \in \mathbb{N}$. Since $i)$ or $ii)$ cannot happen, we get an absurd, finishing the proof.
\end{proof}

\begin{proposition} \label{p5}
	Let $\mu$ satisfy $(\ref{30})$ and $\Lambda_* \geq 1$
	constants given in the Proposition \ref{p3}. Then for $\lambda \geq
	\Lambda_*$, there exists $u_{\lambda}$ a solution of $(P_{\lambda})$
	satisfying $u_{\lambda}\in A_{\mu}^{\lambda}\cap
	I_{\lambda}^{c_{\Gamma}}.$
\end{proposition}
\begin{proof}
	We will suppose, by contradiction, that there are no critical points of
	$I_{\lambda}$ in $A_{\mu}^{\lambda}\cap I_{\lambda}^{c_{\Gamma}}.$
	By Proposition $\ref{p1}$, $I_{\lambda}$ satisfies the $(PS)_d$ condition
	for $d \in [0,c_{\Gamma}] $ and $\lambda$ large enough. Thereby, there exists $d_{\lambda} >0$ such that
	$$
	\left|\left|I'_{\lambda}(u)\right|\right| \geq d_{\lambda},\
	\forall u \in A_{\mu}^{\lambda}\cap I_{\lambda}^{c_{\Gamma}}.
	$$
	On the other hand, by Proposition \ref{p4}, 
	$$
	\|I'_{\lambda}(u)\|
	\geq \sigma_0,\ \forall u \in (A_{2\mu}^{\lambda} \setminus
	A_{\mu}^{\lambda})\cap I_{\lambda}^{c_{\Gamma}},
	$$ 
	where $\sigma_0$ is independent of $\lambda.$ Now, let us define the continuous functions
	$\Psi:E_{\lambda}\rightarrow \mathbb{R}$ and
	$H:I_{\lambda}^{c_{\Gamma}}\rightarrow \mathbb{R}$ by
	\begin{align*}
	\Psi(u)= 1,&  \quad u \in  A_{3\mu/2}^{\lambda},\\
	\Psi(u)=0, & \quad u \not \in A_{2\mu}^{\lambda}, \\
	0 \leq \Psi(u)\leq 1,&\quad \mbox{for}\ u \in E_{\lambda},
	\end{align*}
	and
	$$
	H(u)=\left\{ \begin{array}{cc}
	-\Psi(u)\left|\left|Y(u)\right|\right|^{-1}Y(u),&\ u \in A_{2\mu}^{\lambda},\\
	0,& u\not \in A_{2\mu}^{\lambda},
	\end{array}
	\right.
	$$ 
	where $Y$ is a pseudogradient  vector field for $I_{\lambda}$ on 
	$$
	X=\{ u \in E_{\lambda}; I_{\lambda}(u) \neq 0\}.
	$$ 
	Notice that 
	$$
	\left|\left|H(u)\right|\right| \leq 1,\ \mbox{for all}\ \lambda \geq \Lambda_*\ \mbox{and}\ u \in I_{\lambda}^{c_{\Gamma}}. 
	$$
	The above information ensures the existence of a flow $\eta: [0,+\infty) \times I_{\lambda}^{c_{\Gamma}}$ defined by
	$$
	\left\{ \begin{array}{ccc}
	\dfrac{d\eta(t,u)}{dt}&=& H(\eta(t,u))\\
	\eta(0,u)&=&u\in I_{\lambda}^{c_{\Gamma}},
	\end{array}
	\right.
	$$
	verifying
	\begin{equation}\label{35}
	\dfrac{d I_{\lambda}(\eta(t,u))}{dt}\leq -\Psi(\eta(t,u))\left|\left|I'_{\lambda}(\eta(t,u))\right|\right|\leq 0,
	\end{equation}
	\begin{equation}\label{36}
	\left|\left|\dfrac{d\eta}{dt}\right|\right|= \left|\left|H(\eta)\right|\right|\leq 1,
	\end{equation}
	and
	\begin{equation}\label{37}
	\eta(t,u)=u,\ \forall\ t\geq 0\ \mbox{and}\ u \in I_{\lambda}^{c_{\Gamma}}\setminus A_{2 \mu}^{\lambda}.
	\end{equation}
In what follows, we set
$$
\beta(\overrightarrow{s})=\eta(T,\gamma_0(\overrightarrow{s})),\
	\forall \overrightarrow{s}\in \overline{Q},
$$ 
where $T>0$ will be fixed later on. 

Since 
$$
	\gamma_0(\overrightarrow{s})\not \in A_{2 \mu}^{\lambda},\ \forall
	\overrightarrow{s}\in \partial (\overline{Q}),
$$ 
we deduce that
$$
	\beta(\overrightarrow{s})=\gamma_0(\overrightarrow{s}),\
	\forall \overrightarrow{s}\in \partial(\overline{Q}).
$$ 
Moreover, 	it is easy to check that
$$
I_{\lambda,\mathbb{R}^4\setminus 	\Omega'_{\Gamma}}(\beta(\overrightarrow{s}))\geq 0, \quad \forall \overrightarrow{s}\in \overline{Q},
$$
showing that $\beta \in \Sigma_{\lambda}$.
	
	Note that $supp(\gamma_0(\overrightarrow{s})) \subset
	\overline{\Omega}_{\Gamma}$ for all $\overrightarrow{s} \in
	\overline{Q}$ and  $I_{\lambda}(\gamma_0(\overrightarrow{s}))$
	independent of $\lambda \geq \Lambda.$ Furthermore,
	$$
	I_{\lambda}(\gamma_0(\overrightarrow{s}))\leq c_{\Gamma}, \forall \overrightarrow{s}\in \overline{Q}
	$$
	and
	$$
	I_{\lambda}(\gamma_0(\overrightarrow{s}))= c_{\Gamma},\ \mbox{if}\ s_j=1, \forall j \in \Gamma. 
	$$
	Therefore,
	$$
	m_0=max\left\{ I_{\lambda}(u);u \in \gamma_0(\overline{Q})\setminus A_{\mu}^{\lambda}\right\}< c_{\Gamma},
	$$
	and $m_0$ is independent of $\lambda$.
	
	\begin{claim}\label{cl1}
		There exists a constant $K_*>0$ such that
		$$
		|I_{\lambda_n}(u)-I_{\lambda_n}(v)|\leq K_*\|u-v\|_{\lambda_n}
		$$
		for all $u,v\in \overline{B}_{(M+3)/4}(0)$.
	\end{claim}
	In fact, let $u,v\in\overline{B}_{(M+3)/4}(0)$, there is $K>0$ such that
	$$
	|\langle I_{\lambda_n}'(tu+(1-t)v),w\rangle|\leq K,~~~\forall w\in E_{\lambda_n},~\|w\|_{\lambda_n}\leq 1.
	$$
	Since,
	$$
	|\langle I_{\lambda_n}'(tu+(1-t)v),w\rangle|\leq \dfrac{M+3}{4}+\int_{\R^2}|f(tu+(1-t)v)w|,
	$$
	we only need to prove the boundedness of the above integral. Using the Main Estimate on $f$ ,
	\begin{equation}\label{equa1}
	\int_{\R^2}|g(tu+(1-t)v)w|\leq\dfrac{M+3}{2}+C\int_{\R^2}|w|b_\tau(tu+(1-t)v).
	\end{equation}
	By the H\"older's inequality,
	\begin{equation}\label{equa2}
	\int_{\R^2}|w|b_\tau(tu+(1-t)v)\leq |w|_{q'}|b_\tau(tu+(1-t)v)|_q,
	\end{equation}
	where $1/q+1/q'=1$. Since $M<1$,
	$$
	\|tu+(1-t)v\|_{\lambda_n}\leq t\|u\|_{\lambda_n}+(1-t)\|v\|_{\lambda_n}\leq\dfrac{M+3}{4}<1.
	$$
	Then, we can take $q>1$, $q$ near $1$, such that $
	q\tau(M+3)/4<1$. Thus, from Corollary \ref{cc1}
	\begin{equation}\label{equa3}
	|b_\tau(tu+(1-t)v)|_q\leq C,~~\forall t\in[0,1],~u,v\in B_{(M+3)/4}(0).
	\end{equation}
	Therefore, from $(\ref{equa1})$, $(\ref{equa2})$ and $(\ref{equa3})$,
	$$
	\int_{\R^2}|f(tu+(1-t)v)w|\leq C,~~~~\forall u,v \in \overline{B}_{(M+3)/4}(0),~ t\in[0,1],~\|w\|_\lambda\leq 1,
	$$
	showing that Claim \ref{cl1} holds.
	
	\vspace{0.5 cm}
	
	From Claim $\ref{cl1}$ follows that if $T$ is large enough, the estimate below holds 
	\begin{equation} \label{38}
	\max_{\overrightarrow{s} \in
		\overline{Q})}I_{\lambda}\left(
	\beta(\overrightarrow{s})\right)\leq
	\max\{m_0,c_{\Gamma}-\frac{1}{2K_*}\sigma_0\mu\}.
	\end{equation}
	Indeed, fixing $u=\gamma_0(\overrightarrow{s}) \in E_{\lambda}.$ If $u
	\not \in A_{\mu}^{\lambda},$ we have 
	$$
	I_{\lambda}\left(
	\eta(t,u)\right)) \leq I_{\lambda}\left(
	\eta(0,u)\right))=I_{\lambda}(u)\leq m_0, \quad \forall t \geq 0.
	$$ 
	On the other hand, if $u \in A_{\mu}^{\lambda}$, let us fix $\tilde{\eta}(t)=\eta(t,u). $ Taking
	$\tilde{d}_{\lambda}=\min\{d_{\lambda},\sigma_0\}$ and
	$T=\frac{\sigma_0 \mu}{2 K_*d_{\lambda}}>0,$ we will analyze the following cases: \\
	
	\noindent \textbf{Case 1:} $\tilde{\eta}(t) \in A_{3\mu/2}^{\lambda}, \forall
	t \in [0,T].$ \\
	
	\noindent \textbf{Case 2:} $\tilde{\eta}(t_0) \in \partial
	A_{3\mu/2}^{\lambda},\ \mbox{for some}\ t_0 \in [0,T].$\\

	\noindent {\bf Analysis of the Case 1:} \,  In this case,
	$$
	\Psi(\tilde{\eta}(t))\equiv 1, \quad \forall t \in [0,T].
	$$ 
	and 
	$$
	\| I'_{\lambda}(\tilde{\eta}(t))\| \geq \tilde{d}_{\lambda}, \forall t \in [0,T].
	$$
	Hence,
	$$
	I_{\lambda}(\tilde{\eta}(T))=I_{\lambda}(u)+\int_{0}^{T}\frac{d}{ds}I_{\lambda}(\tilde{\eta}(s))ds \leq c_{\Gamma}-\int_{0}^{T}\tilde{d}_{\lambda}ds,
	$$
	it following that
	$$
	I_{\lambda}(\tilde{\eta}(T))\leq c_{\Gamma}-\tilde{d}_{\lambda}T=c_{\Gamma}-\frac{1}{2K_*}\sigma_0\mu.
	$$
	
	\noindent {\bf Analysis of the Case  2:} \, Let $0 \leq t_1 \leq t_2 \leq T,$ satisfying
	$\tilde{\eta}(t_1)\in \partial A_{\mu}^{\lambda},$
	$\tilde{\eta}(t_2)\in \partial A_{3\mu/2}^{\lambda}$ and
	$\tilde{\eta}(t)\in  A_{3\mu/2}^{\lambda}\setminus
	A_{\mu}^{\lambda}, \forall t\in [t_1,t_2].$ Then
	\begin{equation}\label{39}
	\|\tilde{\eta}(t_1)-\tilde{\eta}(t_2)\|\geq \frac{1}{2K_*}\mu.
	\end{equation}
	Indeed, denoting $w_1=\tilde{\eta}(t)$ and $w_2=\tilde{\eta}(t_2),$
	it follows that 
	$$
	\|w_2\|_{\lambda,\mathbb{R^N}\setminus
		\Omega'_{\Gamma}}=\frac{3}{2} \mu \quad \mbox{or} \quad 
	\left|I_{\lambda,j_0}(w_2)-c_{j_0}\right|=\frac{3}{2} \mu.
	$$
	From definition of $A_{\mu}^{\lambda},$ we have $\|w_2\|_{\lambda,\mathbb{R^N}\setminus\Omega'_{\Gamma}}\leq \mu. $ Thus,
	$$
	\|w_2-w_1\|_{\lambda} \geq \frac{1}{K_*}\left|I_{\lambda,j_0}(w_2)-I_{\lambda,j_0}(w_1)\right|\geq \frac{1}{2K_*}\mu.
	$$
	By Mean Value Theorem
	\begin{equation}\label{40}
	\|\tilde{\eta}(t_1)-\tilde{\eta}(t_2)\|_{\lambda}\le\left|\left|\dfrac{d\eta}{dt}\right|\right|\left| t_1-t_2\right|.
	\end{equation}
	As $\left|\left|\dfrac{d\eta}{dt}\right|\right| \le 1$, from (\ref{39}) and (\ref{40}),   
	$$
	\left| t_1-t_2\right| \ge \frac{1}{2K_*}\mu.
	$$
	Hence
	$$
	I_{\lambda}(\tilde{\eta}(T))\leq I_{\lambda}(u)-\int_{0}^{T}\Psi(\tilde{\eta}(s))\|I'_{\lambda}(\tilde{\eta}(s))\|ds,
	$$
	and so,
	$$
	I_{\lambda}(\tilde{\eta}(T))\leq c_{\Gamma}-\int_{t_1}^{t_2}\sigma_0ds \leq c_{\Gamma}-\frac{\sigma_0}{2K_*}\mu,
	$$
	proving (\ref{38}).
	
	Thereby,  
	$$
	b_{\lambda,\Gamma}\leq \max_{[1/R^2,1]^l}I_{\lambda}(\widehat{\eta}(s_1,\cdots,s_l))\leq \max\{m_0,c_{\Gamma}-\frac{1}{2K*}\sigma_0 \mu\} < c_{\Gamma},
	$$ 
	which is an absurd, because $b_{\lambda, \Gamma}\rightarrow c_{\Gamma},$ when $\lambda \rightarrow \infty.$
\end{proof}

\vspace{0.5 cm}

Thus, we can conclude that $I_{\lambda}$ has a solution $u_{\lambda}$ in $A_{\mu}^{\lambda}$ for $\lambda$ large enough.

\vspace{0.5 cm}

\noindent \textbf{Completion of the Proof of Theorem \ref{T1}:}

\vspace{0.5 cm}

From the Proposition \ref{p5} there exists $\{u_{\lambda_n}\}$ with
$\lambda_n \rightarrow +\infty$ satisfying:
$$I'_{\lambda_n}(u_{\lambda_n})=0,$$
$$
\|u_{\lambda_n}\|_{\lambda_n,\mathbb{R}^N\setminus \Omega'_{\Gamma}} \rightarrow 0
$$
and
$$
I_{\lambda_n,j}(u_{\lambda_n})\rightarrow c_j, \forall j \in \Gamma
$$
Therefore, from of Proposition \ref{p2},
$$
u_{\lambda_n} \rightarrow u \quad \mbox{in} \quad  H^2(\mathbb{R}^4)\ \mbox{with}\ u\ \in H_{0}^{2}(\Omega_{\Gamma}).
$$
Moreover, $u$ is a nontrivial solution of
\begin{equation} \label{PF}
		\left\{ \begin{array}{c}
		\Delta^2 u  +u  =  f(u),\mbox{in}\ \Omega_j \ \\
		u=\dfrac{\partial u}{\partial \eta} =0,\ \mbox{on}\  \partial\Omega_j,
		\end{array}
		\right.
		\end{equation}
with $I_j(u)=c_j$ for all $i \in \Gamma$.  Now, we claim that $u=0$ in $\Omega_j$, for all $j \notin \Gamma$. Indeed, it possible to prove that there is $\sigma_1>0$, which is independent of $j$, such that if $v$ is a  nontrivial solution of (\ref{PF}), then
$$
\|v\|_{H_0^{2}(\Omega_j)} \geq \sigma_1.
$$
However, the solution $u$ verifies 
$$
\|u\|_{H^{2}(\mathbb{R}^{4} \setminus \Omega'_\Gamma)}=0,
$$
showing that $u=0$ in $\Omega_j$, for all $j \notin \Gamma$. This finishes the proof of Theorem \ref{T1}.

\end{document}